\documentclass[]{article} 
\usepackage[utf8]{inputenc}
\usepackage[english]{babel}
\usepackage{color,soul}
\usepackage{graphicx}
\usepackage{graphics}
\usepackage{amssymb,amsmath,amsthm,epsfig}
\newcommand{\norm}[1]{{\left\|{#1}\right\|}}
\newcommand{\ent}[1]{{\left[{#1}\right]}}

\newcommand{\scal}[1]{{\left\langle{#1}\right\rangle}}

\newcommand{\vp}{\psi^{(\alpha)}_{n,c}}

\newtheorem{theorem}{Theorem}
\newtheorem{lemma}{Lemma}
\newtheorem{corollary}{Corollary}

\newtheorem{proposition}{Proposition}

\newtheorem{remark}{Remark}
\newcounter{reh}
\newcounter{rek}
\vspace{5mm}
\numberwithin{equation}{section}
\begin{document}
	\begin{center}
		{\large {\bf Approximation in Hankel Sobolev Space by Circular prolate spheroidal series }}\\
		\vskip 1cm Mourad Boulsane$^*$ {\footnote{
				Corresponding author: Mourad Boulsane, Email: boulsane.mourad@hotmail.fr},\footnote{This work was supported by the DGRST Research Grant LR21ES10 and the PHC-Utique grant 20G1503.}}
	\end{center}
	\vskip 0.5cm {\small
		\noindent $^*$ Carthage University,
		Faculty of Sciences of Bizerte, Department of  Mathematics, Jarzouna, 7021, Tunisia.
	}
	
	\begin{abstract}
	Recently, with the progress of science and the characteristic properties that distinguish the Slepian system called Prolate spheroidal wave functions from the others orthonormal systems, it became clear its important contributions in several areas such as mathematical statistics, signal processing, numerical analysis etc...The main issue of this work is to establish the convergence quality of the truncated error of a function $f$ from the Hankel Sobolev space $\mathcal{H}_{\alpha}^r$, $r > 0$ and $\alpha\geq -1/2$, by a generalized prolate spheroidal wave functions basis in $L^2$ norm.  In the meantime, we will create two uniform approximations of the previous system by two special functions, the first is a Bessel function type and the second is a modified Jacobi polynomial. Furthermore, using the fact that the generalized PSWFs are the eigenfunctions of a Sturm Liouville operator, we will improve a new bound of the associated eigenvalues. The generalized prolate basis we consider here is the circular prolate spheroidal wave function (CPSWFs), called also Hankel prolate, introduced by D. Slepian in 1964, denoted by $\vp$, $c > 0$. The Hankel prolate covers the classical PSWFs $(\alpha=\pm1/2)$ for which corresponding results have been established previously by many authors.
		 
	\end{abstract}
{\it Key words:} Finite Hankel transform operator, eigenfunctions
and eigenvalues, circular prolate spheroidal wave functions.\\[0.3cm]
{\bf 2010 Mathematics Subject Classification.}  42C10, 41A60.
	
	\section{Introduction}
	The problem of approximation by prolate spheroidal wave functions (PSWFs) series is posed strongly in the last decade of this millennium. In fact, they are the best essentially time and band-limited signals with
	bandwidth $c>0$. These properties provide us with good space-time localization tools to be used in signal theory and other fields. Our aim in this paper is to ensure a good approximation quality of a functions $f\in L^2(0,1)$ by its  series expansion in a generalized PSWFs basis called Hankel prolate or circular prolate spheroidal wave functions (CPSWFs) introduced by D. Slepian in his pioneer work \cite{Slepian3}. In reality, we need the quality of the convergence speed where we need some regularity on the desired function, so we will work on the Hankel Sobolev space $\mathcal{H}_{\alpha}^r(0,1)$ that we will more describe it in the sequel. 
	We recall that the circular prolate spheroidal wave functions (CPSWFs) or the 2D Slepian functions are the radial part of the classical band limited  functions defined by D.Slepian on $L^2(B_2)$, where $B_2$ is the unit disk of $\mathbb{R}^2$. More precisely, in \cite{Slepian3}, the author has proved that the CPSWFs, denoted by $\vp$, are the eigenfunctions of the finite Hankel transform $\mathcal{H}_c^{\alpha}$ defined on $L^2(0,1)$ by the associated  Bessel kernel $H_c^{\alpha}(x,y)=\sqrt{cxy}J_{\alpha}(cxy)$ with $\mu_n^{\alpha}(c)$ are the associated eigenvalues. That is
	\begin{equation}\label{eq5}
	\mathcal{H}_c^{\alpha}(\vp)(x)=\int_0^1\sqrt{cxy}J_{\alpha}(cxy)\vp(y)dy=\mu_n^{\alpha}(c)\vp(x) ~~~~~ \forall x\in [0,1],
	\end{equation}
	where $J_{\alpha}$ is the Bessel function of the first type and order $\alpha\geq -1/2$. From the compact self-adjoint properties of $\mathcal{H}_c^{\alpha}$, the 2d-Slepian functions form an orthonormal basis of $L^2(0,1).$ 
	Moreover, they form an orthogonal basis of the Hankel Paley Wiener space $\mathcal {HB}_c^{\alpha}$, the space of functions from $L^2(0,\infty)$ with Hankel transforms supported on $[0,c]$, that is
	\begin{equation}\label{eq0.5}
	\int_0^{\infty}\vp(y)\psi_{m,c}^{(\alpha)}(y)dy=\frac{1}{c(\mu_n^{\alpha}(c))^2}\delta_{n,m},~~~~~
	\mathcal H^{\alpha}(\vp)=\frac{1}{c\mu_{n,\alpha}(c)}\vp(\frac{.}{c})\chi_{[0,c]}.
	\end{equation}
	and
	\begin{equation}\label{HBc}
	\mathcal{H}B^\alpha_c=\left \{ f\in L^2(0,\infty),\,\, \mbox{Support} \mathcal{H}^{\alpha}f\subseteq [0,c] \right\}
	\end{equation}
	Here the Hankel operator $\mathcal{H}^\alpha$ defined on $L^2(0,\infty)$ with kernel $K(x,y)=\sqrt{xy}J_{\alpha}(xy)$. By a lucky incident, D. Slepian has proved that our integral operator $\mathcal{H}_c^{\alpha}$ commutes with a Sturm-Liouville differential operator $\mathcal{L}_c^{\alpha}$ defined on $C^2(0,1)$ by 
	\begin{equation}
	\mathcal{L}_c^{\alpha}(\phi)=\dfrac{d}{dx} \left[ (1-x^2)\dfrac{d}{dx} \phi \right]+\left(\dfrac{\dfrac{1}{4}-\alpha^2}{x^2}-c^2x^2\right)\phi
	\end{equation}
	Consequently,  $\vp$ is the $n-$th order bounded   eigenfunction of the positive self adjoint operator $-\mathcal{L}_c^{\alpha}$  associated with the eigenvalue $\chi_{n}^{\alpha}(c),$ that is
	\begin{equation}
	\label{differ_operator2}
	-\dfrac{d}{dx} \left[ (1-x^2)\dfrac{d}{dx} \vp(x) \right] - \left( \dfrac{\dfrac{1}{4}-\alpha^2}{x^2}-c^2x^2 \right)\vp(x)=\chi_{n}^{\alpha}(c)\vp(x),\quad x\in [0,1].
	\end{equation}
	Further, D.Slepian has given a local estimate for $\chi_n^{\alpha}(c)$ by using the min-max theorem, which given by the following inequality
	\begin{equation}\label{e5}
	(2n+\alpha+1/2)(2n+\alpha+3/2)\leq \chi_n^{\alpha}(c)\leq (2n+\alpha+1/2)(2n+\alpha+3/2)+c^2.
	\end{equation}
	For more details about the CPSWFs, the reader is refereed to \cite{Karoui-Boulsane} and \cite{M.B}.
	
	From the beginning, the question of approximation by a different system such as trigonometric functions, orthogonal polynomials, Bessel functions etc... is an important issue that has been carefully considering in harmonic analysis. In this paper, we are interesting of the Slepian's type bases. In \cite{Karoui-Bonami1} and \cite{W}, the authors had solved this problem in the classical PSWFs basis introduced by D. Slepian and his co-authors H. Landau and H. Pollak in a series of paper in the 1960s. Moreover, this question had solved in the case of weighted Prolate (WPSWFs) basis introduced by wang in \cite{W.Z} and \cite{K.S}. In this work, we will extend the previous studies to the case of the CPSWFs basis. The answer has been proven in section 4 and it is given by the following main result. For every $f\in\mathcal{H}^r_\alpha(0,1)$ and for every $N>\frac{ec}{4}$, we have
	\begin{equation}
	\norm{f-\pi^{\alpha}_{N,c}(f)}_{L^2(0,1)}\leq \frac{C^1_{\alpha}}{\left(\lambda^{(\alpha)}_{N+1}\right)^r}\norm{f}_{\mathcal{H}^r_{\alpha}\left(0,1\right)}+\frac{C^2_{\alpha}}{N^{1/2}}e^{-aN}\norm{f}_{L^2(0,1)}.
	\end{equation}
	Here $a>0$ uniformly constant depending only of $\alpha$, $(\lambda_N^{(\alpha)})_N$ is the sequence of positive zeros of Bessel Function $J_\alpha$ of the first kind and $\pi_{N,c}^\alpha$ is the projection operator over the span of $\{\vp, n\leq N\}$ given by $\pi_{N,c}^\alpha(f)=\sum_{n=0}^Na_n(f)\vp$. It is very important to say that we gratefully acknowledge many important factors that helped us achieve our aim. Specifically, the uniform approximation of the CPSWFs and the quality of the decay rate of its series expansion coefficients in a modified Bessel basis of $L^2(0,1)$ and modified Jacobi polynomials basis described later, contributed greatly to overcoming obstacles.
		
	This work is organized as follows. In section 2, we give some mathematical preliminaries on Hankel Sobolev space and related properties, then we try to define a new Sobolev space related to the differential operator $\mathcal{L}_c^\alpha$ on $C^2(0,1)$. Section 3 is devoted to the uniform approximation of the CPSWFs in terms of two special functions. The first is a modified Bessel function when we will ensure the important approximation in the neighbourhood of $1$. Then, we are going to finish with the uniform approximation by a modified Jacobi polynomial of order $(\alpha,0)$ denoted by $T_n^\alpha$ that we will define it in the sequel. It is important to notice the quality of approximation of the CPSWFs in the neighbourhood of $0$ using the last estimate. We will finish by quoting our main result in section 4, that we will study the quality of approximation of a function $f$ from Hankel Sobolev space by its truncated CPSWFs series .
	\section{Mathematical Preliminaries}
	In this preliminary part, we present a construction of the Hankel Sobolev space with an associated norm, then we create a new Sobolev space associated to a Sturm Liouville differential operator type in order to compare the two associated norms. Finally, we give some properties of the modified Bessel function family $\phi^{(\alpha)}_j$ over $L^2(0,1)$ defined by $\phi^{(\alpha)}_j(x)=\frac{\sqrt{2x}J_{\alpha}(\lambda_j^{(\alpha)}x)}{|J_{\alpha+1}(\lambda_j^{(\alpha)})|}$, where $J_\alpha$ is the Bessel function of the first kind of order $\alpha\geq -1/2$ and $\lambda_j^{(\alpha)}$ is the associated $j-$th order positive zero.
	\subsection{Some facts about Hankel Sobolev space $\mathcal{H}^r_{\alpha}\left(I\right)$}
We recall that the classical Sobolev of index $r>0$ is the subspace of $L^2(I)$ defined by $$H^r\left(I\right)=\{u\in L^2(I);~~~ \forall k\leq r , ~~~u^{(k)}\in L^2(I)\}.$$
When $I=\mathbb{R}$ and by a simple application of Perseval's identity, $H^r$ is equivalent to the following subspace 
$$H^r\left(\mathbb{R}\right)=\{u\in L^2(\mathbb{R});~~~ \int_{\mathbb{R}}(1+x^2)^{r}|\mathcal{F}(u)(x)|^2dx<\infty\},$$ 
where $\mathcal{F}$ is the Fourier transform. Using the same principle, we define the Banach Hankel Sobolev space $\mathcal{H}^r_{\alpha}\left(I\right)$, $I\subset (0,\infty),$ by
\begin{equation}
\mathcal{H}^r_{\alpha}\left(I\right)=\{u\in L^2(I);~~~ x^{(\alpha+k+1/2)}\left(\frac{1}{x}\frac{d}{dx}\right)^{(k)}\left(x^{-(\alpha+1/2)}u\right)\in L^2(I)~\,,\forall k\leq r\}.
\end{equation}
Note that when $I=(0,1)$ and by applying $\left(\frac{1}{x}\frac{\partial}{\partial x}\right)^k$ on \eqref{eq5}, $k\geq 0,$ we obtain
\begin{equation*}
\left(\frac{1}{x}\frac{d}{dx}\right)^{(k)}\left(x^{-(\alpha+1/2)}\vp\right)(x)=\frac{c^{\alpha+1/2}}{\mu_n^{(\alpha)}(c)}\int_0^1\left(\frac{1}{x}\frac{d}{dx}\right)^{(k)}\left(\frac{J_{\alpha}(cxy)}{(cxy)^{\alpha}}\right)y^{\alpha+1/2}\vp(y)dy.
\end{equation*}
From \cite{O.L.B.C}, we have $\left(\frac{1}{x}\frac{d}{dx}\right)^{(k)}\left(\frac{J_{\alpha}(cxy)}{(cxy)^{\alpha}}\right)=(-1)^k(cy)^{2k}\frac{J_{\alpha+k}(cxy)}{(cxy)^{\alpha+k}}$. Then, we get
\begin{eqnarray}
x^{(\alpha+k+1/2)}\left(\frac{1}{x}\frac{d}{dx}\right)^{(k)}\left(x^{-(\alpha+1/2)}\vp\right)&=&\frac{(-c)^k}{\mu_n^{(\alpha)}(c)}\mathcal{H}_c^{\alpha+k}(x^k\vp).
\end{eqnarray}
Here $\mathcal{H}_c^{\alpha}$ is the finite Hankel transform. From the previous details and the fact that the family of CPSWFs forms an orthonormal system in $L^2(0,1),$ we are sure that the later is a subset of $\mathcal{H}^r_{\alpha}\left(0,1\right).$ 
Moreover, when $I= (0,\infty)$ and taking into account the self-adjoint property of Hankel operator $\mathcal{H}^{\alpha}$ satisfying $\mathcal{H}^{\alpha}\circ\mathcal{H}^{\alpha}=id_{L^2(0,\infty)}$ with the use of \cite{O.L.B.C}, we get
\begin{eqnarray*} x^{(\alpha+k+1/2)}\left(\frac{1}{x}\frac{d}{dx}\right)^{(k)}\left(x^{-(\alpha+1/2)}u\right)(x)&=&x^{(\alpha+k+1/2)}\int_0^{\infty}\left(\frac{1}{x}\frac{d}{dx}\right)^{(k)}\left(\frac{J_{\alpha}(xy)}{(xy)^{\alpha}}\right)y^{\alpha+1/2}\mathcal{H}^{\alpha}(u)(y)dy\\&=&(-1)^k\int_0^{\infty}\sqrt{xy}J_{\alpha+k}(xy)y^k\mathcal{H}^{\alpha}(u)(y)dy\\&=&(-1)^k\mathcal{H}^{\alpha+k}(x^k\mathcal{H}^{\alpha}(u))(x).
\end{eqnarray*}
By using Perseval's identity, we can finally check that
$$x^{(\alpha+k+1/2)}\left(\frac{1}{x}\frac{d}{dx}\right)^{(k)}\left(x^{-(\alpha+1/2)}u\right)\in L^2(0,\infty)\Leftrightarrow \mathcal{H}^{\alpha+k}(x^k\mathcal{H}^{\alpha}(u))\in L^2(0,\infty)\Leftrightarrow x^k\mathcal{H}^{\alpha}(u)\in L^2(0,\infty).$$
Hence $\mathcal{H}^r_{\alpha}(0,\infty)$ is equivalent to
\begin{equation}
\mathcal{H}^r_{\alpha}\left(0,\infty\right)=\{u\in L^2(0,\infty);~~~ \int_0^{\infty}(1+y^2)^r|\mathcal{H}^{\alpha}(u)(y)|^2dy<\infty\}
\end{equation} 
endowed with the associated norm defined by $\norm{u}^2_{\mathcal{H}^r_{\alpha}\left((0,\infty)\right)}=\int_0^{\infty}(1+y^2)^r|\mathcal{H}^{\alpha}(u)(y)|^2dy$. \\Consequently, from \eqref{eq0.5} and the fact that the CPSWFs family form an orthonormal basis of $L^2(0,1)$ we have for every $n\in\mathbb{N},$ $\psi_{n,c}^{\alpha}\in\mathcal{H}^r_{\alpha}\left(0,\infty\right).$ The previous  space contains more than the prolates, we notice that any function with compactly supported Hankel transform is an element of the later. We conclude that the Hankel Paley Wiener space $\mathcal {HB}_c^{\alpha}$ given in \eqref{HBc} is a subspace of $\mathcal{H}^r_{\alpha}\left(0,\infty\right).$\\
In order to improve the results of our paper and based on \cite{W.Z}, we need to introduce the Sturm Liouville Sobolev space $\widetilde{\mathcal{H}}^r_{\alpha}\left(I\right)$ associated to the differential operator $\mathcal{L}_c^{\alpha}$. Firstly, we introduce the bilinear form $\mathcal{K}_c^{\alpha}$ defined on $E=\{\varphi\in C^2(0,1); \,\varphi(0)=0\}$ by
\begin{eqnarray*}
\mathcal{K}_c^{\alpha}(\varphi,\psi)&=&\scal{-\mathcal{L}_c^{\alpha}(\varphi),\psi}_{L^2((0,1)}\\&=&\scal{\varphi',\psi'}_{L^2((0,1),\omega)}+c^2\scal{x\varphi,x\psi}_{L^2(0,1)}-(\frac{1}{4}-\alpha^2)\scal{\frac{\varphi}{x},\frac{\psi}{x}}_{L^2(0,1)}.
\end{eqnarray*}
where $\omega(x)=(1-x^2).$ The Sturm Liouville operator $-\mathcal{L}_c^{\alpha}$ is positive, symmetric and self-adjoint. We define the fractional operator $\left(-\mathcal{L}_c^{\alpha}\right)^{1/2}$ satisfying the following property
\begin{equation}
\mathcal{K}_c^{\alpha}(\varphi,\varphi)=\norm{\left(-\mathcal{L}_c^{\alpha}\right)^{1/2}(\varphi)}^2_{L^2(0,1)}
\end{equation}
which imply that for every $m\in\mathbb{N}$, we have 
\begin{equation}\label{eq0.6}
\begin{cases}
\norm{\left(-\mathcal{L}_c^{\alpha}\right)^{m+1/2}(\varphi)}^2_{L^2(0,1)}=\mathcal{K}_c^{\alpha}\Big(\left(-\mathcal{L}_c^{\alpha}\right)^m(\varphi),\left(-\mathcal{L}_c^{\alpha}\right)^m(\varphi)\Big)
\\
\norm{\left(-\mathcal{L}_c^{\alpha}\right)^{m+1}(\varphi)}^2_{L^2(0,1)}=\mathcal{K}_c^{\alpha}\Big(\left(-\mathcal{L}_c^{\alpha}\right)^m(\varphi),\left(-\mathcal{L}_c^{\alpha}\right)^{m+1}(\varphi)\Big).
\end{cases}
\end{equation}
Let $m\in\mathbb{N}^*,$ we define the Sturm Liouville Sobolev space on $L^2(0,1)$ by
\begin{equation}
\widetilde{\mathcal{H}}^m_{\alpha}\left(I\right)=\{\varphi\in L^2(0,1);~~~\left(-\mathcal{L}_c^{\alpha}\right)^{\frac{m}{2}}(\varphi)\in L^2(0,1)\}.
\end{equation}
To the space $\widetilde{\mathcal{H}}^m_{\alpha}\left(I\right),$ we define the associated norm $$\norm{\varphi}^2_{\widetilde{\mathcal{H}}^m_{\alpha}\left(I\right)}=\norm{\left(-\mathcal{L}_c^{\alpha}\right)^{\frac{m}{2}}(\varphi)}^2_{L^2(0,1)}=\mathcal{K}_c^{\alpha}\Big(\left(-\mathcal{L}_c^{\alpha}\right)^{\frac{m-1}{2}}(\varphi),\left(-\mathcal{L}_c^{\alpha}\right)^{\frac{m-1}{2}}(\varphi)\Big).$$
In the same context, we improved more the definition of our space, so for a function $\varphi\in L^2(0,1)$ developed in CPSWFs basis and the fact that $-\mathcal{L}_c^{\alpha}(\vp)=\chi_n^{\alpha}(c)\vp$, we have
$$\left(-\mathcal{L}_c^{\alpha}\right)^m(\varphi)=\sum_{n=0}^{\infty}a_n(\varphi)\left(\chi_n^{\alpha}(c)\right)^m\vp~~~~~~~~a_n(\varphi)=\int_0^1\varphi(x)\vp(x)dx.$$
It follows from \eqref{eq0.6}, by orthogonality and the fact that $\mathcal{K}_c^{\alpha}(\vp,\psi_{m,c}^{(\alpha)})=\chi_n^{\alpha}(c)\delta_{n,m}$
\begin{equation}
\begin{cases}
\norm{\left(-\mathcal{L}_c^{\alpha}\right)^m(\varphi)}^2_{L^2(0,1)}=\displaystyle\sum_{n=0}^{\infty}\left(\chi_n^{\alpha}(c)\right)^{2m}|a_n(\varphi)|^2 \\ \norm{\left(-\mathcal{L}_c^{\alpha}\right)^{m+\frac{1}{2}}(\varphi)}^2_{L^2(0,1)}=\displaystyle\sum_{n=0}^{\infty}\left(\chi_n^{\alpha}(c)\right)^{2m+1}|a_n(\varphi)|^2
\end{cases}.
\end{equation}
We conclude now from the two last identities that
\begin{equation}
\widetilde{\mathcal{H}}^{2m}_{\alpha}\left(I\right)=\{\varphi\in L^2(0,1);~~~\displaystyle\sum_{n=0}^{\infty}\left(\chi_n^{\alpha}(c)\right)^{2m}|a_n(\varphi)|^2<\infty\}
\end{equation}
and
\begin{equation}
\widetilde{\mathcal{H}}^{2m+1}_{\alpha}\left(I\right)=\{\varphi\in L^2(0,1);~~~\displaystyle\sum_{n=0}^{\infty}\left(\chi_n^{\alpha}(c)\right)^{2m+1}|a_n(\varphi)|^2<\infty\}.
\end{equation} 
Finally, by space interpolation method, we introduce the Hilbert space $\widetilde{\mathcal{H}}^{r}_{\alpha}\left(I\right)$, $r>0$, given by $$\widetilde{\mathcal{H}}^{r}_{\alpha}\left(I\right)=\{\varphi\in L^2(0,1);~~~\norm{\varphi}^2_{\widetilde{\mathcal{H}}^{r}_{\alpha}\left(I\right)}=\norm{\left(-\mathcal{L}_c^{\alpha}\right)^{\frac{r}{2}}(\varphi)}^2_{L^2(0,1)}=\displaystyle\sum_{n=0}^{\infty}\left(\chi_n^{\alpha}(c)\right)^{r}|a_n(\varphi)|^2<\infty\}.$$
\subsection{Some facts about $\phi^{(\alpha)}_j$}
Let $(\lambda_j^{(\alpha)})_{j\geq 1}$ be the set of the positive zeros of Bessel function $J_{\alpha}$ of the first kind and order $\alpha\geq -1/2$ satisfies the following estimate given in \cite{Elbert} by $\lambda_j^{(\alpha)}\sim \pi j+\frac{\pi}{2}(\alpha-\frac{1}{2})$ when $(j\to\infty)$. It's well known that the family $\phi^{(\alpha)}_j$ defined on $L^2(0,1)$ by $$\phi^{(\alpha)}_j(x)=\frac{\sqrt{2x}J_{\alpha}(\lambda_j^{(\alpha)}x)}{|J_{\alpha+1}(\lambda_j^{(\alpha)})|},\,\,x\in(0,1)$$ forms an orthonormal system, that is
\begin{equation}
\int_0^1xJ_{\alpha}(\lambda_i^{(\alpha)}x)J_{\alpha}(\lambda_j^{(\alpha)}x)dx=\frac{|J_{\alpha+1}(\lambda_j^{(\alpha)})|^2}{2}\delta_{i,j}.
\end{equation}
Moreover, from \cite{H}, the previous system forms an orthonormal basis of $L^2(0,1)$ and by a lucky coincident, the span of $(\phi^{(\alpha)}_j)$ is a subspace of $\mathcal{H}^r_{\alpha}\left(0,1\right).$ Indeed, from \cite{O.L.B.C}, we have for any integer $k\geq 0$
\begin{eqnarray}\label{e1}
x^{(\alpha+k+1/2)}\left(\frac{1}{x}\frac{d}{dx}\right)^{(k)}\left(x^{-(\alpha+1/2)}\phi^{(\alpha)}_j\right)&=&(\lambda_j^{(\alpha)})^{\alpha}\frac{\sqrt{2}x^{(\alpha+k+1/2)}}{|J_{\alpha+1}(\lambda_j^{(\alpha)})|}\left(\frac{1}{x}\frac{d}{dx}\right)^{(k)}\left(\frac{J_{\alpha}(\lambda_j^{(\alpha)}x)}{(\lambda_j^{(\alpha)}x)^{\alpha}}\right)\nonumber\\&=&(-1)^k(\lambda_j^{(\alpha)})^{\alpha+2k}\frac{\sqrt{2}x^{(\alpha+k+1/2)}}{|J_{\alpha+1}(\lambda_j^{(\alpha)})|}\frac{J_{\alpha+k}(\lambda_j^{(\alpha)}x)}{(\lambda_j^{(\alpha)}x)^{\alpha+k}}\nonumber\\&=&(-\lambda_j^{(\alpha)})^k\frac{|J_{\alpha+k+1}(\lambda_j^{(\alpha+k)})|}{|J_{\alpha+1}(\lambda_j^{(\alpha)})|}\phi^{(\alpha+k)}_j\left(\frac{\lambda_j^{(\alpha)}}{\lambda_j^{(\alpha+k)}}x\right).
\end{eqnarray}
Then, for every $k\geq 0$, we have $$x^{(\alpha+k+1/2)}\left(\frac{1}{x}\frac{d}{dx}\right)^{(k)}\left(x^{-(\alpha+1/2)}\phi^{(\alpha)}_j\right)=(-\lambda_j^{(\alpha)})^k\frac{|J_{\alpha+k+1}(\lambda_j^{(\alpha+k)})|}{|J_{\alpha+1}(\lambda_j^{(\alpha)})|}\phi^{(\alpha+k)}_j\left(\frac{\lambda_j^{(\alpha)}}{\lambda_j^{(\alpha+k)}}x\right)\in L^2(0,1).$$ On the other hand, $\phi^{(\alpha)}_j$ is a solution of a Sturm Liouville differential equation, see \cite{Watson}, given by
\begin{equation}
V''+\left((\lambda_j^{(\alpha)})^2+\frac{1/4-\alpha^2}{x^2}\right)V=0.
\end{equation}
Hence, $(\phi^{(\alpha)}_j)$ are the eigenfunctions of the positive self adjoint differential operator $(-\mathcal{L})$ defined on $C^2(0,1)$ by  $(-\mathcal{L})(V)=-V''-\left(\frac{1/4-\alpha^2}{x^2}\right)V$ with associated eigenvalues $((\lambda_j^{(\alpha)})^2)$. In the following and by using the same techniques given in \cite{N}, we can prove that our norm defined on $\mathcal{H}^r_{\alpha}\left(0,1\right)$ by $$\norm{u}^2_{\mathcal{H}^r_{\alpha}\left(0,1\right)}=\sum_{k=0}^r\norm{x^{(\alpha+k+1/2)}\left(\frac{1}{x}\frac{d}{dx}\right)^{(k)}\left(x^{-(\alpha+1/2)}u\right)}^2_{L^2(0,1)}$$ is equivalent to the following norm defined on $\mathcal{H}^r_{\alpha}\left(0,1\right)$ by $$\norm{u}_{\alpha,r}^2=\norm{\left(-\mathcal{L}\right)^r(u)}^2_{L^2(0,1)}=\sum_{j=1}^{+\infty}\lambda_j^{2r}b^2_j(u)$$	where $b_j(u)=\int_0^1\phi^{(\alpha)}_j(x)u(x)dx$, that is

	\begin{equation}\label{e2}
	\norm{u}_{\alpha,r}\sim \norm{u}_{\mathcal{H}^r_{\alpha}\left(0,1\right)}.
	\end{equation}
\section{Uniform Approximation of the $\mathcal{L}_c^{\alpha}-$spectrum}
The uniform approximation problem of CPSWFs has been studied in \cite{K.M}, but a precise estimate is not solved on the whole interval $(0,1)$. So in this section we improve the uniform approximation estimate of the CPSWFs in terms of Bessel function of the first kind and order $0$, then in terms of Jacobi polynomials of order $(\alpha,0)$ on $(0,1)$. In the meantime, we deduce a new bound of the eigenvalues $\chi_n^{\alpha}(c)$ which we will use it for the proof of our main result. Let us specify before starting that the various details given in the following paragraph are strongly inspired from the classic case given in \cite{Karoui-Bonami2} and \cite{K.S}.
\subsection{Uniform Estimate in terms of Bessel functions}
It is known that the CPSWFs satisfy a Sturm-Liouville differential equation given by \eqref{differ_operator2}, that is
\begin{equation}\label{e11}
\frac{\partial}{\partial x}\left[(1-x^2)\vp\right]-\frac{(\alpha^2-1/4)}{x^2}\vp=-\chi_n^{\alpha}(c)(1-qx^2)\vp,\,\,x\in(0,1),
\end{equation}
where $q=\frac{c^2}{\chi_n^{\alpha}(c)}<1$. Using the Liouville transform given by $U(S(x))=(1-x^2)^{1/4}(1-qx^2)^{1/4}\vp(x)$, where $S(x)=\displaystyle\int_x^1\sqrt{\frac{1-qt^2}{1-t^2}}dt,\,\,x\in(0,1)$, $U$ satisfies the following second order differential equation 
\begin{equation}\label{e6}
U''(S(x))+\left(\chi_n^{\alpha}(c)+\theta_{\alpha}(S(x))\right)U(S(x))=0,\,\,x\in(0,1).
\end{equation}
Here $$(\theta_{\alpha}\circ S)(x)=\frac{(1/4-\alpha^2)}{x^2(1-qx^2)}+(1-qx^2)^{-1}\varphi^{-1}(x)\frac{d}{dx}\left[(1-x^2)\varphi'(x)\right]$$ and $$\varphi(x)=(1-x^2)^{-1/4}(1-qx^2)^{-1/4}.$$ For more details about Liouville transformation see \cite{V.S.M}. A straightforward computation gives us an explicit expression for $\theta_{\alpha}\circ S$ that is for $Q(x)=(1-x^2)(1-qx^2)$ and from the fact that $\displaystyle\frac{\varphi'}{\varphi}=-\frac{1}{4}\frac{Q'}{Q},$ we have
\begin{eqnarray}\label{e10}
	\theta_{\alpha}\circ S(x)&=&\frac{(1/4-\alpha^2)}{x^2(1-qx^2)}+\frac{1}{16(1-qx^2)}\left[(1-x^2)\left(\frac{Q'}{Q}\right)^2-4\frac{d}{dx}\left[(1-x^2)\frac{Q'}{Q}\right]\right]\\&=&\frac{(1/4-\alpha^2)}{x^2(1-qx^2)}+\frac{1}{2(1-qx^2)}+\frac{x^2}{4(1-qx^2)(1-x^2)}\nonumber\\&+&\frac{q}{2(1-qx^2)^2}+q\frac{x^2}{4(1-qx^2)^2}-\frac{5q(1-q)x^2}{4(1-qx^2)^3}\nonumber\\&=&\frac{(1/4-\alpha^2)}{x^2(1-qx^2)}+\theta\circ S(x)\nonumber.
\end{eqnarray}
 Note that $\theta\circ S$ is the classical case given in \cite{Karoui-Bonami2} satisfying the following inequality
 \begin{equation}\label{e12}
 \left|\theta\circ S(x)-\frac{1}{4S^2(x)}\right|\leq\frac{3+2q}{4}\frac{1}{(1-qx^2)^2},\,~~~x\in(0,1).
 \end{equation}
 
  Note that we can transform the equation \eqref{e6} in the radial Schr\"odinger equation form, that is \begin{equation*}
U''(s)=\left(-\chi_n^{\alpha}(c)-\frac{1}{4s^2}+g_{\alpha}(s)\right)U(s) \,~~ s\in(0,S(0)),
\end{equation*}
where $g_{\alpha}(s)=\displaystyle\frac{1}{4s^2}-\theta_{\alpha}(s).$ By using Olver's techniques see \cite{Olver}, $U$ has the following form
\begin{equation}
U(s)=A_n^{\alpha}\sqrt{\sqrt{\chi_n^{\alpha}(c)}s}\,J_{0}(\sqrt{\chi_n^{\alpha}(c)}s)+\sqrt{s}\,\varepsilon_{n,c}^{\alpha}(s) \,~~ s\in(0,S(0)),
\end{equation}
with
\begin{equation} \left|\sqrt{s}\varepsilon_{n,c}^{\alpha}(s)\right|\leq \sqrt{s}\frac{M_{0}(\sqrt{\chi_n^{\alpha}(c)}s)}{E_{0}(\sqrt{\chi_n^{\alpha}(c)}s)}\left[\exp(G_{\alpha,n}(s))-1\right],~~s\in(0,S(0)).\end{equation} Here $G_{\alpha,n}(s)=\frac{\pi}{2}\displaystyle\int_0^stM_{0}^2\left(\sqrt{\chi^{\alpha}_n(c)}t\right)|g_{\alpha}(t)|dt$ and $M_{\alpha}$, $E_{\alpha}$ are defined as follows. Let $X_{\alpha}$ be the first positive zero of the equation $J_{\alpha}(x)+Y_{\alpha}(x)=0$, then 
\begin{equation*}
M_{\alpha}(x)=\begin{cases}
\left(2|Y_{\alpha}(x)|J_{\alpha}(x)\right)^{1/2} &\mbox{ if } 0\leq x\leq X_{\alpha}\\\left(J^2_{\alpha}(x)+Y^2_{\alpha}(x)\right)^{1/2} &\mbox{ if } x\geq X_{\alpha}\end{cases}\mbox{and}~
E_{\alpha}(x)=\begin{cases}
-\left(\frac{Y_{\alpha(x)}}{J_{\alpha}(x)}\right)^{1/2} &\mbox{ if } 0\leq x\leq X_{\alpha}\\ 1 &\mbox{ if } x\geq X_{\alpha}	\end{cases}.
\end{equation*}
Hence,
\begin{equation}\label{p1}
 \displaystyle\sqrt{x}\frac{M_{\alpha}(x)}{E_{\alpha}(x)}=\begin{cases}
		\sqrt{2x}J_{\alpha}(x)&\mbox{ if } 0\leq x\leq X_{\alpha}\\\sqrt{x}\left(J^2_{\alpha}(x)+Y^2_{\alpha}(x)\right)^{1/2} &\mbox{ if } x\geq X_{\alpha}
		\end{cases}
			\end{equation}
			and
		\begin{equation*}
		 \displaystyle\sup_{x>0}xM^2_{\alpha}(x)\leq m_{\alpha}=\begin{cases}
		\frac{2}{\pi} &\mbox{ when } |\alpha|\leq 1/2 \\ \mbox{ finite and larger than $\frac{2}{\pi}$}  &\mbox{ when }|\alpha|\geq 1/2   
		\end{cases}.
	\end{equation*}
	In order to estimate our rest $\sqrt{s}\varepsilon_{n,c}^{\alpha}(s),$ we need to estimate before $G_{\alpha,n}\circ S$ which is given in the following lemma.
\begin{lemma}
	Let $\alpha\geq-1/2$ and $c>0$ then for every $0\leq q\leq 1/2$, we have 
	\begin{equation}
	\left|G_{\alpha,n}(S(x))\right|\leq\frac{2}{(1-q)\sqrt{\chi_{n}^{\alpha}(c)}}\left(\frac{|\alpha^2-1/4|}{x^2}+(3+2q)\right)\frac{\sqrt{1-x^2}}{\sqrt{1-qx^2}},\,~~x\in(0,1).
	\end{equation}
\end{lemma}
\begin{proof}
Recall that $g_{\alpha}\circ S$ has the following form $g_{\alpha}\circ S(x)=\displaystyle\frac{1}{4S^2(x)}-\theta_{\alpha}\circ S(x)$, which we can rewrite in the following form 
$$g_{\alpha}\circ S(x)=\displaystyle\frac{1}{4S^2(x)}-\theta\circ S(x)-\frac{(1/4-\alpha^2)}{x^2(1-qx^2)}.$$ 
From \eqref{e12}, we can easily check that 
\begin{equation}
\left|g_{\alpha}(S(x))\right|\leq\frac{|\alpha^2-1/4|}{x^2(1-qx^2)}+\frac{(3+2q)}{4(1-qx^2)^2}.
\end{equation}
It follows from \eqref{p1} and the previous inequality that 
\begin{eqnarray*}
\left|G_{\alpha,n}(S(x))\right|&\leq& \frac{|\alpha^2-1/4|}{\sqrt{\chi_{n}^{\alpha}(c)}}\int_x^1\frac{dy}{y^2(1-y^2)^{1/2}(1-qy^2)^{1/2}}+\frac{(3+2q)}{\sqrt{\chi_{n}^{\alpha}(c)}}\int_x^1\frac{dy}{(1-y^2)^{1/2}(1-qy^2)^{3/2}}dy\\&\leq&\frac{1}{\sqrt{\chi_{n}^{\alpha}(c)}}\left(\frac{|\alpha^2-1/4|}{x^2}+(3+2q)\right)\int_x^1\frac{dy}{(1-y^2)^{1/2}(1-qy^2)^{3/2}}dy.
\end{eqnarray*}
Using the following inequality from \cite{Karoui-Bonami2},  
$$\int_x^1\frac{dy}{(1-y^2)^{1/2}(1-qy^2)^{3/2}}dy=\frac{1}{1-q}\left(\frac{qx\sqrt{1-x^2}}{\sqrt{1-qx^2}}+S(x)\right)\leq\frac{2}{1-q}\frac{\sqrt{1-x^2}}{\sqrt{1-qx^2}}, $$
we conclude that
\begin{eqnarray*}
\left|G_{\alpha,n}(S(x))\right|&\leq&\frac{2}{(1-q)\sqrt{\chi_{n}^{\alpha}(c)}}\left(\frac{|\alpha^2-1/4|}{x^2}+(3+2q)\right)\frac{\sqrt{1-x^2}}{\sqrt{1-qx^2}}.
\end{eqnarray*}	
\end{proof}
By using the previous analysis, we obtain the following main result that generalizes a similar result in \cite{Karoui-Bonami2}.
\begin{theorem}
	Let $\alpha\geq-1/2$ and $c>0,$ then for every integer $n\in\mathbb{N}$ such that $q=\frac{c^2}{\chi_n^{\alpha}(c)}\leq 1/2$ and $x\in(0,1)$, we have 
	\begin{equation}
	\vp(x)=A_n^{\alpha}\frac{\sqrt{\sqrt{\chi_n^{\alpha}(c)}S(x)}J_{0}(\sqrt{\chi_n^{\alpha}(c)}S(x))}{(1-x^2)^{1/4}(1-qx^2)^{1/4}}+R_{n,c}^{\alpha}(S(x)),
	\end{equation}
	where $\left|R_{n,c}^{\alpha}(S(x))\right|\leq\displaystyle\frac{2}{(1-q)\sqrt{\chi_{n}^{\alpha}(c)}}\left(\frac{|\alpha^2-1/4|}{x^2}+(3+2q)\right)\displaystyle\frac{(1-x^2)^{1/4}}{(1-qx^2)^{3/4}}.$
\end{theorem}
The following proposition ensure the uniform boundedness of $U\circ S=(1-x^2)^{1/4}(1-qx^2)^{1/4}\vp$ over $(0,1)$.
\begin{proposition}
	Let $\alpha\geq 1/2$ and $c>0$ be two fixed real numbers such that $c^2>\alpha^2-1/4$, then for every integer $n\in\mathbb{N}$ such that $q=\frac{c^2}{\chi_n^{\alpha}(c)}\leq 1/2$, there exists $x_n^*\in(0,1)$ satisfying the following estimate
	\begin{equation}\label{e3}
	\sup_{x\in(0,1)}\left|U\left(S(x)\right)\right|=\left|U\left(S(x_n^*)\right)\right|\leq\left[\left|U(S(\delta_n^1))\right|^2+\frac{1}{\chi_n^{\alpha}(c)+\theta_\alpha(S(\delta_n^1))}\left|U'(S(\delta_n^1))\right|^2\right]^{1/2},
	\end{equation}
	where 
	\begin{equation}\label{e14}
	\delta_n^1=\sqrt{\frac{2(\alpha^2-1/4)}{\chi^{\alpha}_n(c)}}<x_n^*<\delta_n^2=\frac{\pi+\frac{\pi}{2}\alpha-\frac{3}{4}}{\sqrt{\chi_n^\alpha(c)+1/4-\alpha^2}}.
	\end{equation}
\end{proposition}
\begin{proof}
From a straightforward computation and the monotonicity of $\theta\circ S$ given in \cite{Karoui-Bonami2}, we conclude that $\theta_{\alpha}\circ S$ given by \eqref{e10} is increasing on $(0,1).$ Then from the decay of $S$ we have $\theta_{\alpha}$ is decreasing in $(0,S(0))$. Using Butlewski's theorem, the local maxima of $|U|,$ the solution of \eqref{e6}, are increasing in $(0,S(0))$. Hence, there exists $t_n^*\in(0,S(0))$ such that $\displaystyle\sup_{s\in(0,S(0))}|U(s)|=|U(t_n^*)|$. From the fact that $S$ is invertible and decreasing, there exists $x_n^*\in(0,x_n^{1})$ such that $$\displaystyle\sup_{x\in(0,1)}|U(S(x))|=|U(S(x_n^*))|$$
	where $ x_n^{1}$ is the first positive zero of $\vp$ on $(0,1).$ We give an upper and a lower bound of $x_n^*.$ The upper bound given in \eqref{e14} is proved in \cite{Karoui-Boulsane}. Let's focus on the lower bound.  
	We rewrite \eqref{e6} in the following form 
	\begin{eqnarray*}
	\frac{d}{dx}\left(U'(S(x))\right)&=&\Big(\chi_n^{\alpha}(c)+\theta_\alpha(S(x))\Big)\frac{(1-qx^2)^{1/2}}{(1-x^2)^{1/2}}U(S(x))\\&=&\left(\chi_n^{\alpha}(c)+\frac{1/4-\alpha^2}{x^2(1-qx^2)}+\theta\circ S(x)\right)\frac{(1-qx^2)^{1/2}}{(1-x^2)^{1/2}}U(S(x)).
	\end{eqnarray*}
Therefore, under the positivity condition of    $\delta_n(x)=\chi_n^{\alpha}(c)+\frac{1/4-\alpha^2}{x^2(1-qx^2)}+\theta\circ S(x)$ and the fact that $U'(S(x_n^*))=0$, we conclude that $U\circ S$ and its derivate have the same positive sign over $(0,x_n^*)$.
Indeed, from the positivity of $\theta\circ S,$ if $x^2(1-qx^2)>\frac{\alpha^2-1/4}{\chi^{\alpha}_n(c)}$ then $\delta_n>0$. More precisely, we can easily check that $$x^2(1-qx^2)>\frac{\alpha^2-1/4}{\chi^{\alpha}_n(c)}\Leftrightarrow \sqrt{\frac{1}{2q}\left(1-\sqrt{1-4q\frac{\alpha^2-1/4}{\chi^{\alpha}_n(c)}}\right)}<x<1.$$ Then, we conclude that $\sqrt{\frac{1}{2q}\left(1-\sqrt{1-4q\frac{\alpha^2-1/4}{\chi^{\alpha}_n(c)}}\right)}\lesssim\delta_n^1=\sqrt{\frac{2(\alpha^2-1/4)}{\chi^{\alpha}_n(c)}}<x_n^*.$ 
  
It is easy to check that for every $t<\varepsilon_n^\alpha=S(\delta_n^\alpha)$, $a(t)=\chi_n^\alpha(c)+\theta_\alpha(t)>0$, using the fact that $\theta_\alpha\circ S$ is increasing and $a(\varepsilon_n^\alpha)>0.$ Multiply \eqref{e6} by $U'$ and integrate between $t$ and $\varepsilon_n^\alpha$, we obtain
\begin{equation}
\int_t^{\varepsilon_n^\alpha}a(s)U(s)U'(s)ds=\frac{(U')^2(t)}{2}-\frac{(U')^2(\varepsilon_n^\alpha)}{2}
\end{equation}
From an integration by parts, we have
\begin{equation}
\int_t^{\varepsilon_n^\alpha}a'(s)\frac{(U)^2(s)}{2}ds=a(\varepsilon_n^\alpha)\frac{(U)^2(\varepsilon_n^\alpha)}{2}-a(t)\frac{(U)^2(t)}{2}-\int_t^{\varepsilon_n^\alpha}a(s)U(s)U'(s)ds
\end{equation}
It follow that
\begin{eqnarray*}
\frac{(U')^2(t)}{2}+a(t)\frac{(U)^2(t)}{2}&=&a(\varepsilon_n^\alpha)\frac{(U)^2(\varepsilon_n^\alpha)}{2}+\frac{(U')^2(\varepsilon_n^\alpha)}{2}-\int_t^{\varepsilon_n^\alpha}a'(s)\frac{(U)^2(s)}{2}ds\\&=& C_n^\alpha+\int_t^{\varepsilon_n^\alpha}\left(-\frac{a'(s)}{a(s)}\right)a(s)\frac{(U)^2(s)}{2}ds
\end{eqnarray*}
 Finally, for every $0\leq t<\varepsilon_n^\alpha,$ we have
 \begin{equation*}
 a(t)\frac{(U)^2(t)}{2}\leq C_n^\alpha+\int_t^{\varepsilon_n^\alpha}\left(-\frac{a'(s)}{a(s)}\right)a(s)\frac{(U)^2(s)}{2}ds. 
 \end{equation*}
 From the Gr\"onwall Lemma and the positivity of $a$, we get 
 \begin{equation*}
 U^2(t)\leq 2\frac{C_n^\alpha}{a(\varepsilon_n^\alpha)}=\left|U(\varepsilon_n^\alpha)\right|^2+\frac{1}{\chi_n^{\alpha}(c)+\theta_\alpha(\varepsilon_n^\alpha)}\left|U'(\varepsilon_n^\alpha)\right|^2,\,\,~ 0\leq t<\varepsilon_n^\alpha. 
 \end{equation*}
 It remains to take $t=S(x_n^*)$ then we get our result.
\end{proof}
The following theorem improves the Slepian's bounds of the eigenvalues $\chi_n^{\alpha}(c)$ given in \eqref{e5}.
\begin{theorem}
	Let $c>0$ and $\alpha\geq-1/2$ be two fixed real numbers, then for every $\varepsilon>0,$ there exist $n_0\geq 0$ such that for every $n\geq n_0,$ we have
	\begin{equation}\label{e16}
	(2n+\alpha+\frac{1}{2})(2n+\alpha+\frac{3}{2})+(1/4-\varepsilon)c^2\leq \chi_{n}^{\alpha}(c)\leq (2n+\alpha+\frac{1}{2})(2n+\alpha+\frac{3}{2})+c^2
	\end{equation}
\end{theorem}
\begin{proof}
	In order to prove the previous result, we need the derivative with respect to $c$ of the differential equation satisfied by the CPSWFs, that is
	$$(1-x^2)\partial_c\left[(\psi_{n,c}^{(\alpha)})''\right]-2x\partial_c\left[(\psi_{n,c}^{(\alpha)})'\right]+\left(\chi_{n}^{\alpha}(c)+\frac{1/4-\alpha^2}{x^2}-c^2x^2\right)\partial_c\left[(\psi_{n,c}^{(\alpha)})\right]+\left(\partial_c\chi_{n}^{\alpha}(c)-2cx^2\right)\psi_{n,c}^{(\alpha)}=0.$$
	We rewrite the last equation in the following form
	\begin{equation}\label{eq0}
	\left(\mathcal{L}_{c}^{\alpha}+\chi_{n}^{\alpha}(c)Id\right)\left(\partial_c(\psi_{n,c}^{(\alpha)})\right)+\left(\partial_c\chi_{n}^{\alpha}(c)-2cx^2\right)\psi_{n,c}^{(\alpha)}=0.
	\end{equation}
	Using the self-adjoint property of $\mathcal{L}_{c}^{\alpha}$ and the fact that $\psi_{n,c}^{(\alpha)}$ are the eigenfunctions of the last, one gets
	$$\scal{\left(\mathcal{L}_{c}^{\alpha}+\chi_{n}^{\alpha}(c)Id\right)\left(\partial_c(\psi_{n,c}^{(\alpha)})\right),\psi_{n,c}^{(\alpha)}}_{L^2(0,1)}=0.$$
	Using \eqref{eq0} and the previous equation, we obtain
	$$\scal{\left(\partial_c\chi_{n}^{(\alpha)}(c)-2cx^2\right)\psi_{n,c}^{(\alpha)},\psi_{n,c}^{(\alpha)}}_{L^2(0,1)}=0.$$	
	Finally, from the fact that $\norm{\psi_{n,c}^{(\alpha)}}_{L^2(0,1)}=1,$ one gets
	\begin{equation}\label{eq1}
	\partial_c\chi_{n}^{\alpha}(c)=2c\int_0^1x^2\left(\psi_{n,c}^{(\alpha)}(x)\right)^2dx.
	\end{equation}
	Let $B=\int_0^1x^2\left(\psi_{n,c}^{(\alpha)}(x)\right)^2dx,$ from the series expansion of $x^2\vp$ in $(T_k^\alpha)$ basis of $L^2(0,1)$ given by
	\begin{equation}\label{e17}
	T_n^{\alpha}(x)=\sqrt{2(2n+\alpha+1)}x^{\alpha+1/2}P_n^{(\alpha,0)}(1-2x^2),
	\end{equation}
	 we get
	\begin{eqnarray*}
	B=\norm{x\vp}^2_{L^2(0,1)}&\geq&\norm{x^2\vp}^2_{L^2(0,1)}=\displaystyle\sum_{k=0}^\infty\scal{x^2\vp,T_k^{\alpha}}^2_{L^2(0,1)}\\&\geq&\scal{x^2\vp,T_n^{\alpha}}^2_{L^2(0,1)}.
	\end{eqnarray*}
By using the three recursion formula satisfied by $T_n^\alpha$ inspired from $\cite{Slepian3}$, such that
\begin{equation*}
x^2T_n^\alpha=\gamma_{n,\alpha}^0T_n^\alpha-(\gamma_{n+1,\alpha}^1T_{n+1}^\alpha+\gamma_{n,\alpha}^{1}T_{n-1}^\alpha)
\end{equation*}
where $\gamma_{n,\alpha}^0=\frac{1}{2}\left[1+\frac{\alpha^2}{(2n+\alpha)(2n+\alpha+2)}\right]$ and $\gamma_{n,\alpha}^1=\frac{n(n+\alpha)}{(2n+\alpha)\sqrt{(2n+\alpha+1)(2n+\alpha-1)}}$ with the $k-$th series expansion coefficient of $\vp$ in $(T_k^\alpha)$ basis denoted by  $d_k^n=\scal{\vp,T_k^{\alpha}}_{L^2(0,1)}$ for $k=n, n+1, n-1,$ we obtain
\begin{eqnarray*}
B&\geq&\scal{x^2\vp,T_n^{\alpha}}^2_{L^2(0,1)}\\&=&\left(\gamma_{n,\alpha}^0d_n^n-(\gamma_{n+1,\alpha}^1d_{n+1}^n+\gamma_{n,\alpha}^{1}d_{n-1}^n)\right)^2.
\end{eqnarray*}
Using the fact that $d_n^n=1-\eta_n^n$ such that $0\leq\eta_n^n\lesssim 1/n^2$, see \cite{Boulsane-Jaming-Souabni}, we rewrite the last inequality under the following form
\begin{eqnarray*}
B&\geq&\left(\gamma_{n,\alpha}^0-(\eta_n^n\gamma_{n,\alpha}^0+\gamma_{n+1,\alpha}^1d_{n+1}^n+\gamma_{n,\alpha}^{1}d_{n-1}^n)\right)^2\\&\geq&\left(\gamma_{n,\alpha}^0\right)^2-2\gamma_{n,\alpha}^0\left|\eta_n^n\gamma_{n,\alpha}^0+\gamma_{n+1,\alpha}^1d_{n+1}^n+\gamma_{n,\alpha}^{1}d_{n-1}^n\right|\\&\geq&\frac{1}{4}-\kappa_n.
\end{eqnarray*}
From the decay rate of $d_k^n$, $k= n-1$ and $k=n+1$ given in \cite{Boulsane-Jaming-Souabni}, we get $\kappa_n=O(1/n).$
    
	Finally, to conclude it suffices to integrate \eqref{eq1} on $(0,c)$ and use the fact that $\chi_{n}^{\alpha}(0)=(2n+\alpha+\frac{1}{2})(2n+\alpha+\frac{3}{2})$. 
\end{proof}
\subsection{Uniform Approximation in terms of Jacobi polynomials}
In order to improve the uniform approximation of the CPSWFs in terms of $T_n^\alpha$ given in \eqref{e17}, we first give the differential equation satisfied by $K_n^\alpha(x)=\frac{T_n^{\alpha}(x)}{x^{\alpha+1/2}}$ which we write
\begin{equation}
\frac{d}{dx}\left[x^{2\alpha+1}(1-x^2)\frac{d}{dx}[K_n^{\alpha}]\right]+\gamma_n^{\alpha}(0)x^{2\alpha+1}K_n^{\alpha}=0,
\end{equation}
where $\gamma_n^\alpha(c)=\chi_n^\alpha(c)+1/4-(\alpha+1)^2$ for $c\geq 0.$ 
We notice that we can rewrite the equation \eqref{e11} satisfied by $\vp$ in the following form
\begin{equation}
\frac{d}{dx}\left[x^{2\alpha+1}(1-x^2)\frac{d}{dx}[V]\right]+\gamma_n^{\alpha}(0)x^{2\alpha+1}V=g(x)x^{2\alpha+1}V,
\end{equation}
where $V=\frac{\vp}{x^{\alpha+1/2}}$ and $g(x)=\gamma_n^{\alpha}(0)-\gamma_n^{\alpha}(c)+c^2x^2$ satisfying $|g(x)|\leq c^2$ for every $x\in (0,1)$.

We denote by $H_n^{\alpha}=\sqrt{(2n+\alpha+1)}Q_n^{(\alpha,0)}(1-2x^2)$ the second linearly independent solution of the previous homogeneous differential equation given by the Jacobi function of the second kind, then through the use of {\bf WKB} method and from the unboundedness of $Q_n^{(\alpha,0)}$ in the neighborhood of $1$, we can check that for every $x\in(0,1)$ the bounded solution of \eqref{e11} satisfies the following equality
\begin{equation}
\vp(x)=A_n^{\alpha}T_n^{\alpha}(x)+x^{\alpha+1/2}\int_0^x\frac{S_n^{\alpha}(x,y)}{\mathcal{W}(K_n^{\alpha},H^{\alpha}_n)(y)}g(y)\vp(y)y^{\alpha+1/2}dy
\end{equation}
where $S_n^{\alpha}(x,y)=K_n^{\alpha}(x)H_n^{\alpha}(y)-K_n^{\alpha}(y)H_n^{\alpha}(x)$ and $\mathcal{W}(K_n^{\alpha},H^{\alpha}_n)$ is the Wronskian given by the following formula from \cite{E},
\begin{equation}
\mathcal{W}(K_n^{\alpha},H^{\alpha}_n)(y)=\frac{-(2n+\alpha+1)}{y^{2\alpha+1}(1-y^2)}.
\end{equation}
Finally, we have the following estimate
\begin{equation}
\vp(x)=A_n^{\alpha}T_n^{\alpha}(x)-\frac{x^{\alpha+1/2}}{2n+\alpha+1}\int_0^xS_n^{\alpha}(x,y)g(y)\vp(y)(1-y^2)y^{3(\alpha+1/2)}dy
\end{equation}
	Using the same techniques given in \cite{Karoui-Bonami2} see also \cite{K.S}, one can easily check that for a fixed $\alpha\geq-1/2$, there exists a constant $C_{\alpha}$ such that  
	\begin{equation}
	y^{2\alpha+2}(1-y^2)\left|S_n^{\alpha}(x,y)\right|\leq C_{\alpha},\,\,~~0\leq y\leq x\leq 1.
	\end{equation}
	By collecting every details together and based on the same techniques given in \cite{Karoui-Bonami2}, we get the following proposition.
\begin{proposition}
	Let $\alpha\geq -1/2$ and $c>0$ then for every integer $n\in\mathbb{N}$, we have 
	\begin{equation}\label{e13}
	\vp(x)=A_n^{\alpha}T_n^{\alpha}(x)+x^{\alpha+1/2}R_{n,c}^{\alpha}(x),\,\,~~x\in(0,1)
	\end{equation}
	where $\left|R_{n,c}^{\alpha}(x)\right|\leq C_{\alpha}\displaystyle\frac{c^2}{2n+\alpha+1}x^{\alpha}$ and $A_n^{\alpha}$ is a constant satisfying the following estimate
	$$\left|A_n^{\alpha}-1\right|\leq\frac{C_{\alpha}}{\sqrt{4\alpha+2}}\frac{c^2}{2n+\alpha+1}.$$
\end{proposition}

\section{Quality of approximation by CPSWFs in Hankel Sobolev Space}
In this section, we study the quality of approximation by the series expansion in CPSWFs basis of a function from the Hankel sobolev space. First, we give a description of the decay rate of the series
expansion coefficients of the eigenfunctions $\vp$ in a generalized Jacobi functions basis $\left(T_k^{\alpha}\right)_k$ of $L^2(0,1)$ where will be the connection with the decay of the $\vp-$ series expansion coefficients in $\phi_j^{(\alpha)}$ basis of $L^2(0,1)$, then we conclude in the final subsection by gives the quality of the convergence speed of the truncated series expansion of a function $f$ to $f.$	
\subsection{Local estimate of $d^n_{k}$}
By substitution the series expansion of $\vp$ with respect to the basis $T_k^{\alpha}$ in problem \eqref{differ_operator2}, that is
\begin{equation}
\vp=\sum_{k=0}^{\infty}d_k^n T_k^{\alpha} ~~~~~~~~ d_k^n=\int_0^1T_k^{\alpha}(x)\vp(x)dx,
\end{equation}
 we obtain the three term recurrence relation satisfied by the
coefficients $d_k^n$. More precisely, from \cite{Slepian3}, we have 
\begin{equation}\label{eq3}
f(k,n,c,\alpha)d_k^n=a_{k,\alpha}d_{k-1}^n+a_{k+1,\alpha}d_{k+1}^n,~ \forall k\ge 0,
\end{equation}
where $d_{-1}^{n}=0$ and
\begin{eqnarray}
\label{coeff-CPSWFs}
f(k,n,c,\alpha)&=&\displaystyle\chi_{n}^{\alpha}(c)-(\alpha+2k+\frac{1}{2})(\alpha+2k+\frac{3}{2})-c^2b_{k,\alpha} \\
a_{k,\alpha}&=&\frac{k(k+\alpha)}{(\alpha+2k)\sqrt{\alpha+2k+1}\sqrt{\alpha+2k-1}}c^2 \nonumber \\
b_{k,\alpha}&=&\frac{1}{2}\ent{\frac{\alpha^2}{(\alpha+2k+1)(\alpha+2k)}+1} \nonumber.
\end{eqnarray}
\begin{proposition}
Let $\alpha\geq -1/2$ and $c>0$,
then for every integers $ n\geq 0, k\geq 0$ such that $\chi_{n}^{\alpha}(c)\geq 4(k+1)(k+2+\alpha)+C_{\alpha},$ we have
\begin{equation}
\left|\left(\frac{1}{x}\frac{d}{dx}\right)^k(U)(0)\right|\leq (\chi_{n}^{\alpha}(c))^{k}\left|U(0)\right|.
\end{equation}
where $C_{\alpha}=(\alpha+1/2)(\alpha+3/2)$ and $U(x)=\frac{\vp(x)}{x^{\alpha+1/2}}.$
\end{proposition}

\begin{proof}
To simplify the expressions, we denote by $\psi=\psi_{n,c}^{\alpha}$ and $\chi=\chi_{n,c}^{(\alpha)}$. Using \eqref{differ_operator2} and a straightforward computation, we can easily check that $U(x)=\frac{\psi(x)}{x^{\alpha+1/2}}$ satisfies the following differential equation
\begin{equation}
x^2(1-x^2)\left(\frac{1}{x}\frac{d}{dx}\right)^2(U)(x)=f_1(x)\left(\frac{1}{x}\frac{d}{dx}\right)(U)(x)+f_2(x)U(x)~~~~~ x\in(0,1)
\end{equation}
where $$f_1(x)=2\left((2+\alpha)x^2-(1+\alpha)\right),\,\,f_2(x)=\left(c^2x^2+C_{\alpha}-\chi\right)$$ and $C_{\alpha}=\alpha^2+2\alpha+\frac{3}{4}.$
We denote by $U_k=\left(\frac{1}{x}\frac{d}{dx}\right)^k(U)(0),$ then by applying the operator $\left(\frac{1}{x}\frac{d}{dx}\right)$ to the last equation successively and since $\left(\frac{1}{x}\frac{d}{dx}\right)^j(f_1)=\left(\frac{1}{x}\frac{d}{dx}\right)^j(f_2)=0$ for every $j\geq 2$, $U_k$ satisfies the following three terms recursion formula, for every $k\geq 0$
$$\beta_k^{\alpha}U_{k+2}=\gamma_k^{\alpha}U_{k+1}+\delta_k^{\alpha}U_{k}$$ where
\begin{eqnarray*}
\beta_k^{\alpha}&=& 2(k+2+\alpha).\\
\gamma_k^{\alpha}&=&4(k+1)(k+2+\alpha)+C_{\alpha}-\chi.\\
\delta_k^{\alpha}&=&2c^2(k+1).
\end{eqnarray*}
By induction, it's easy to check that for every $k\geq 0$,
$U_{k+1}U_k<0,$ which gives us the following relation:
$$|U_{k+2}|=\frac{-\gamma_k^{\alpha}}{\beta_k^{\alpha}}|U_{k+1}|+\frac{\delta_k^{\alpha}}{\beta_k^{\alpha}}|U_{k}|.$$
It follows that
$$m_{k+2}=\frac{1}{\beta_{k}^{\alpha}}\left((1-\frac{4(k+1)(k+2+\alpha)+C_{\alpha}}{\chi})m_{k+1}+\frac{2q}{\chi}(k+1)m_{k}\right),$$

where $m_k=\frac{|U_k|}{\chi^k}$ and $q=\frac{c^2}{\chi}.$\\
It remains to prove the following result which is done by induction on $k\in\mathbb{N}$. 
For every $k\geq 0,$ we have, $m_k\leq m_0$.
\end{proof}
\begin{corollary}
	Let $c>0$ and $\alpha\geq-1/2$, then for every $n,k\geq 0$ satisfying the same condition of the previous proposition, we have
	\begin{equation}\label{eq6}
0\leq 	\int_0^1t^{2k+\alpha+1/2}\vp(t)dt\leq \frac{2^{\alpha+k}\Gamma(\alpha+k+1)}{c^{\alpha+1/2}}\frac{|\mu_n^{\alpha}(c)|}{q^k}|U(0)|,
	\end{equation}
	where $U(0)=\displaystyle\lim_{x\to 0^+}\left(\frac{\psi_{n,c}^{\alpha}(x)}{x^{\alpha+1/2}}\right)$ and $q=\frac{c^2}{\chi_n^{\alpha}(c)}$.
\end{corollary}
\begin{proof}
	At first, from \eqref{eq5} we have 
	\begin{eqnarray*}\label{eq2}
		\int_0^1\frac{J_{\alpha}(cxy)}{(cxy)^{\alpha}}y^{\alpha+1/2}\psi_{n,c}^{\alpha}(y)dy&=&\frac{1}{(cx)^{\alpha+1/2}}\mathcal{H}_c^{\alpha}(\psi_{n,c}^{\alpha})(x)\\&=&\mu_n^{\alpha}(c)\frac{\psi_{n,c}^{\alpha}(x)}{(cx)^{\alpha+1/2}}.
	\end{eqnarray*}
	Then, by applying the operator $\left(\frac{1}{x}\frac{d}{dx}\right)^k$ to the last equation, we get
	\begin{equation}
	\int_0^1\left(\frac{1}{x}\frac{d}{dx}\right)^k\left(\frac{J_{\alpha}(cxy)}{(cxy)^{\alpha}}\right)y^{\alpha+1/2}\psi_{n,c}^{\alpha}(y)dy=\mu_n^{\alpha}(c)\left(\frac{1}{x}\frac{d}{dx}\right)^k\left(\frac{\psi_{n,c}^{\alpha}(x)}{(cx)^{\alpha+1/2}}\right).
	\end{equation}
	From \cite{O.L.B.C}, we have  $J_{\alpha+k}(x)=\displaystyle\frac{x^{\alpha+k}}{2^{\alpha+k}\Gamma(\alpha+k+1)}+O(x^{\alpha+k+2})$ in the neighbourhood of zero. Furthermore, from the same source, we have 
	$$\left(\frac{1}{x}\frac{d}{dx}\right)^k\left(\frac{J_{\alpha}(cxy)}{(cxy)^{\alpha}}\right)=(-1)^k(cy)^{2k}\frac{J_{\alpha+k}(cxy)}{(cxy)^{\alpha+k}}.$$ 
	By combining the two last estimates and used in \eqref{eq2} and by using the previous proposition, one gets the desired result \eqref{eq6}.
\end{proof}
\begin{lemma}
Let $c>0,$  $\alpha\geq -1/2$, then for every $n, k\geq 0$ such that $\chi_{n}^{\alpha}(c)\geq(\alpha+2k+\frac{1}{2})(\alpha+2k+\frac{3}{2})+\left(\frac{2\sqrt{\alpha+1}}{(\alpha+2)\sqrt{\alpha+3}}+b_{0,\alpha}\right)c^2$, then  $d_k^n$ and  $d_0^n$ have the same sign .	
	\end{lemma}
\begin{proof}
We recall that from \eqref{eq3}, we have $$d_1^n=\frac{(\alpha+2)\sqrt{\alpha+3}}{c^2\sqrt{\alpha+1}}\left(\chi_{n}^{\alpha}(c)-(\alpha+\frac{1}{2})(\alpha+\frac{3}{2})-c^2b_{0,\alpha}\right)d_0^n.$$
Without loss of the generality, we assume that $d_0^n\geq 0$ then $d_1^n\geq d_0^n$. We assume that for a fixed $k\geq 1$ and for every $1\leq j\leq k,$ $d_{j}^n\geq d_{j-1}^n$ then, from \eqref{eq3}, we have 
\begin{equation}\label{eq4}
c^2\frac{\sqrt{\alpha+1}}{(\alpha+2)\sqrt{\alpha+3}}\left(d_{k-1}^n+d_{k+1}^n\right)\geq \left(\displaystyle\chi_{n}^{\alpha}(c)-(\alpha+2k+\frac{1}{2})(\alpha+2k+\frac{3}{2})-c^2b_{0,\alpha}\right)d_k^n.
\end{equation}
If we suppose that $d_{k+1}^n\leq d_{k}^n$, we end up with a contradiction from \eqref{eq4} with our condition. Hence, we conclude that the suitable $(d_k^n)_k$ is increasing, which imply that $d_k^n\geq 0.$
In the similar manners, we can prove that if $d_0^n\leq 0,$ then the sequence $(d_k^n)_k$ is decreasing and we get $d_k^n\leq 0.$
\end{proof}
\begin{theorem}
	Let $c>0$ and $\alpha>0$ two real numbers, then under the same previous conditions on $n$ and $j$ where $j\leq n,$ we have
	\begin{equation}\label{eq0.6}
\left|d_j^n\right|\leq \frac{C_{\alpha}}{\sqrt{2(2j+\alpha+1)}}\left(\frac{8}{q}\right)^{j}\Gamma(\alpha+j+1)n^{\alpha}|\mu_n^{\alpha}(c)|.	
	\end{equation} 
\end{theorem}
\begin{proof}
	We first give an explicit expression of the 2j-th order moment of $t^{\alpha+1/2}T_j^{\alpha}$ from \cite{O.L.B.C}, that is
	\begin{eqnarray*}
	\alpha_{jj}&=&\int_0^1t^{2j+\alpha+1/2}T_j^{\alpha}(t)dt=\sqrt{2(2j+\alpha+1)}\frac{j!\,\Gamma(j+\alpha+1)}{2\Gamma(2j+\alpha+1)}.
	\end{eqnarray*}
Then, by using Batir inequality given in \cite{B.N}, we obtain
\begin{eqnarray}\label{eq0.7}
\frac{1}{\alpha_{jj}}&=&\frac{1}{\sqrt{2(2j+\alpha+1)}}\frac{2\Gamma(2j+\alpha+1)}{j!\,\Gamma(j+\alpha+1)}\\&\leq&C_{\alpha}\frac{2^{2j}}{\sqrt{2(2j+\alpha+1)}}.\nonumber
\end{eqnarray}
	On the other hand, from the following expansion of $t^{2j+\alpha+1/2}$ in $T_k^{\alpha}$ basis of $L^2(0,1)$ given by $t^{2j+\alpha+1/2}=\displaystyle\sum_{k=0}^{j}\alpha_{jk}T_k^{\alpha}(t)$, we can see that $$\int_0^1t^{2j+\alpha+1/2}\vp(t)dt=\sum_{k=0}^{j}\alpha_{jk}d_k^n.$$
	Moreover, from the positivity of $\alpha_{jk}$ and the previous Lemma, we have
	\begin{equation}\label{eq0.8}
	d_j^n\leq\frac{1}{\alpha_{jj}}\sum_{k=0}^{j}\alpha_{jk}d_k^n=\frac{1}{\alpha_{jj}}\int_0^1t^{2j+\alpha+1/2}\vp(t)dt.
	\end{equation}
	We recall that from \eqref{e13}, we have $\vp\sim A_n^{\alpha}T_n^{\alpha}(x)$ in the neighbourhood of zero. Hence,
	\begin{equation}\label{eq1.0}
	 U(0)=\lim_{x\to 0}A_n^{\alpha}T_n^{\alpha}(x)=A_n^{\alpha}\frac{(\alpha+1)_n}{n!}\sim n^{\alpha}.
	 \end{equation}
	Finally, by using the inequalities \eqref{eq6}, \eqref{eq0.7}, \eqref{eq0.8} and \eqref{eq1.0}, we get 
	\begin{eqnarray*}
	d_j^n&\leq&C_{\alpha}\frac{(2)^{2j}}{\sqrt{2(2j+\alpha+1)}}\frac{2^{\alpha+j}\Gamma(\alpha+j+1)}{c^{\alpha+1/2}}\frac{|\mu_n^{\alpha}(c)|}{q^j}|U(0)|\\&\leq&\frac{C_{\alpha}}{\sqrt{2(2j+\alpha+1)}}\left(\frac{8}{q}\right)^{j}\Gamma(\alpha+j+1)n^{\alpha}|\mu_n^{\alpha}(c)|.
	\end{eqnarray*} 
\end{proof}
\subsection{Estimates for the truncation errors}
Let's start by the following result where the proof is done in a similar way as the proof of Wang and Zhang in \cite{W.Z} since the Hankel operator have the same properties as the Fourier transform.
	For any $\alpha\geq -1/2$, $c>0$ and for any function $u\in \widetilde{\mathcal H}^r_{\alpha}(0,1)$ with $r\geq 0$, we have
\begin{equation}
\norm{\pi^{\alpha}_{N,c}(u)-u}_{L^2(0,1)}\leq \left(\chi_{N+1}^{\alpha}(c)\right)^{-\frac{r}{2}}\norm{u}_{\widetilde{H}^r_{\alpha}(0,1)}.
\end{equation}
Moreover, for any $0\leq \delta\leq r$, we have
\begin{equation}
\norm{\pi^{\alpha}_{N,c}(u)-u}_{\widetilde{H}^r_{\alpha}(0,1)}\leq \left(\chi_{N+1}^{\alpha}(c)\right)^{\frac{\delta -r}{2}}\norm{u}_{\widetilde{H}^r_{\alpha}(0,1)}
\end{equation}
\begin{remark}
	It's well known that L.L.Wang compares in \cite{W} the topology of the two classical subspaces $\widetilde{\mathcal{H}}^r_{c}\left(I\right)$ and $\mathcal{H}^r\left(I\right)$, satisfy 
	$$\norm{\varphi}_{\widetilde{\mathcal{H}}^{r}_{c}\left(I\right)}\leq C\,(1+c^2)^{r/2}\norm{\varphi}_{\mathcal{H}^r\left(I\right)}.$$
	The previous inequality is similar in our case, so we have an advantage when $(1+c^2)$ takes small values but when it is comparable to $\chi_{N+1}^{\alpha}(c)$ we lose our estimate in $\mathcal{H}^r\left(I\right).$ In the sequel, we will ensure that we can give an estimate of the truncation errors uniformly in $c>0$. So we first give a description of the decay rate of the CPSWFs expansion coefficients in $\{\phi_j^{(\alpha)}\}_{j\geq 1}$ basis of $L^2(0,1)$ in the following lemma, which will help us a lot in solving our problem. 
\end{remark}
In the sequel, using \eqref{e16}, we can guarantee a positive integer $N_{\alpha}(n)$ uniformly in $c$ such that for every $k\leq N_{\alpha}(n),$ we have $$\chi_{n}^{\alpha}(c)\geq(\alpha+2k+\frac{1}{2})(\alpha+2k+\frac{3}{2})+\left(\frac{2\sqrt{\alpha+1}}{(\alpha+2)\sqrt{\alpha+3}}+b_{0,\alpha}\right)c^2,$$
For more details, see for example remark $1$ in \cite{Karoui-Bonami1}. 
\begin{lemma}
	Let $c>0$ and $\alpha>0$ two real numbers, then for every $n\in\mathbb{N}$ such that $n> \frac{ec}{4}$,  we have
	\begin{equation}\label{eq0.9}
	\left|\scal{\phi^{(\alpha)}_j,\vp}_{L^2(0,1)}\right|\leq \frac{M}{\lambda_j} e^{-an},
	\end{equation}
	for every $j$ such that $\lambda_j<N_\alpha(n)$. Here $M>0$ and $a>0$ two real numbers uniformly in $c$. 
\end{lemma}
\begin{proof}
	For the simplicity, we denote by $\lambda_j=\lambda_j^{(\alpha)}$ the $j-$th order positive zero of the Bessel function $J_\alpha.$ Using the classical transform given in \cite{Boulsane-Jaming-Souabni} by $\mathcal{H}^{\alpha}\left(\vp(\frac{.}{c})\chi_{(0,c)}\right)=c\mu_n^{\alpha}(c)\vp$ on $(0,\infty)$, we have
	\begin{equation} \scal{\phi^{(\alpha)}_j,\vp}_{L^2(0,1)}=\sqrt{\frac{2}{\lambda_j}}\frac{\mathcal{H}^{\alpha}\left(\vp(\frac{.}{c})\chi_{(0,c)}\right)(\frac{\lambda_j}{c})}{c|J_{\alpha+1}(\lambda_j)|}=\sqrt{\frac{2}{\lambda_j}}\frac{\mu_n^{\alpha}(c)}{|J_{\alpha+1}(\lambda_j)|}\vp(\frac{\lambda_j}{c}).
	\end{equation}
	Further, by using the series expansion of $\vp$ in spherical Bessel basis $j_{k,c}^{(\alpha)}$ defined by $$j_{k,c}^{(\alpha)}(x)=\sqrt{2(2k+\alpha+1)}\frac{J_{2k+\alpha+1}(cx)}{\sqrt{cx}},\,\,x\in (0,\infty)$$ 
	given by $\vp=\frac{|\mu_n^{\alpha}(c)|}{\sqrt{c}\mu_n^{\alpha}(c)}\displaystyle\sum_{k=0}^\infty(-1)^kd_k^nj_{k,c}^{(\alpha)}$, 
	we obtain
	\begin{eqnarray*}
	\scal{\phi^{(\alpha)}_j,\vp}_{L^2(0,1)}&=&\sqrt{\frac{2}{\lambda_j}}\frac{|\mu_{n}^{\alpha}(c)|}{\sqrt{c}|J_{\alpha+1}(\lambda_j)|}\sum_{k=0}^\infty(-1)^kd_k^nj_{k,c}^{(\alpha)}(\frac{\lambda_j}{c})\\&=&\frac{2}{\lambda_j|J_{\alpha+1}(\lambda_j)|}\frac{|\mu_n^{\alpha}(c)|}{\sqrt{c}}\sum_{k=0}^{\infty}(-1)^k\sqrt{(2k+\alpha+1)}d_k^nJ_{2k+\alpha+1}(\lambda_j)\\&=&\frac{2}{\lambda_j|J_{\alpha+1}(\lambda_j)|}\frac{|\mu_n^{\alpha}(c)|}{\sqrt{c}}\sum_{k=0}^{N_{\alpha}(n)}(-1)^k\sqrt{(2k+\alpha+1)}d_k^nJ_{2k+\alpha+1}(\lambda_j)\\&+&\frac{2}{\lambda_j|J_{\alpha+1}(\lambda_j)|}\frac{|\mu_n^{\alpha}(c)|}{\sqrt{c}}\sum_{k=N_{\alpha}(n)+1}^{\infty}(-1)^k\sqrt{(2k+\alpha+1)}d_k^nJ_{2k+\alpha+1}(\lambda_j)\\&=&I_1(n)+I_2(n).
	\end{eqnarray*}
For the first part, we must use the Olenko inequality $\displaystyle\sup_{x\in(0,\infty)}|\sqrt{x}J_{\alpha}(x)|\leq C_{\alpha}$ given in \cite{A.YA. OLenko}. Furthermore, by using $|\lambda_j^{1/2}J_{\alpha+1}(\lambda_j)|^{-1}=O(1)$ and \eqref{eq0.6}, we obtain
\begin{eqnarray*}
|I_1^n|&\leq& C_{\alpha}n^{\alpha}\frac{|\mu_n^{\alpha}(c)|^2}{\lambda_j}\sum_{k=0}^{N_{\alpha}(n)}\Gamma(\alpha+k+1)\left(\frac{8}{q}\right)^{k}\\&\leq& C_{\alpha}n^{\alpha}\frac{|\mu_n^{\alpha}(c)|^2}{\lambda_j}\sum_{k=0}^{N_{\alpha}(n)}\left(\frac{8(k+\alpha+1/2)}{qe}\right)^{k+\alpha+1/2}\\&\leq&C_{\alpha}n^{\alpha}\frac{|\mu_n^{\alpha}(c)|^2}{\lambda_j}\sum_{k=0}^{N_{\alpha}(n)}\left(\frac{8(N_{\alpha}(n)+\alpha+1/2)}{qe}\right)^{k+\alpha+1/2}\\&\leq&C_{\alpha}n^{\alpha}\frac{|\mu_n^{\alpha}(c)|^2}{\lambda_j}\left(\frac{8(N_{\alpha}(n)+\alpha+1/2)}{qe}\right)^{N_{\alpha}(n)+\alpha+3/2}.
\end{eqnarray*}
Using the super-exponential decay rate of $\mu_n^{\alpha}(c)$ given in \cite{M.B}, we conclude that $|I_1^n|\leq M_1 e^{-an}.$\\
For the second part, we just use the fact that $|d_k^n|\leq 1$ and the classical inequality $|J_{\alpha}(x)|\leq \frac{x^{\alpha}}{2^{\alpha}\Gamma(\alpha+1)}$ with the same argument $|\lambda_j^{1/2}J_{\alpha+1}(\lambda_j)|^{-1}=O(1)$ and Batir inequality, then we obtain
\begin{eqnarray*}
|I_2^{n}|&\leq& \frac{C_{\alpha}}{\lambda_j}\frac{|\mu_n^{\alpha}(c)|}{\sqrt{c}}\sum_{k=N_{\alpha}(n)+1}^{\infty}\sqrt{(2k+\alpha+1)}\frac{(\lambda_j)^{2k+\alpha+3/2}}{2^{2k+\alpha+1}\Gamma(2k+\alpha+2)}\\&\leq& \frac{C_{\alpha}}{\lambda_j}\frac{|\mu_n^{\alpha}(c)|}{\sqrt{c}}\left(\frac{e\lambda_j}{4N_{\alpha}(n)+2\alpha+5}\right)^{2N_{\alpha}(n)+\alpha+5/2}\\&\leq&\frac{C_{\alpha}}{\lambda_j}\frac{|\mu_n^{\alpha}(c)|}{\sqrt{c}}\left(\frac{eN_{\alpha}(n)}{4N_{\alpha}(n)+2\alpha+5}\right)^{2N_{\alpha}(n)+\alpha+5/2}
\end{eqnarray*}
\end{proof}
By using the previous analysis, we obtain the following main theorem that provides us with the quality of approximation by the CPSWFs series expansion of functions from $\mathcal{H}_\alpha^r(0,1).$
\begin{theorem}
	Let $c>0$ and $\alpha>0$ two real numbers, then for every integer $N>\frac{ec}{4}$ and for every $f\in \mathcal{H}^r_{\alpha}(0,1)$, we have
	\begin{equation}\label{e15}
	\norm{f-\pi^{\alpha}_{N,c}(f)}_{L^2(0,1)}\leq \frac{C^1_{\alpha}}{\lambda^{r}_{N+1}}\norm{f}_{\mathcal{H}^r_{\alpha}\left(0,1\right)}+\frac{C^2_{\alpha}}{N^{1/2}}e^{-aN}\norm{f}_{L^2(0,1)}
	\end{equation}
\end{theorem}
\begin{proof}
	Let $g$ be the projection of $f$ over span of $\{\phi^{(\alpha)}_j; \, 1\leq j\leq N\}$, then we have
	\begin{eqnarray*}
	\norm{f-\pi^{\alpha}_{N,c}(f)}_{L^2(0,1)}&\leq& \norm{f-g}_{L^2(0,1)}+\norm{g-\pi^{\alpha}_{N,c}(g)}_{L^2(0,1)}+\norm{\pi^{\alpha}_{N,c}(g)-\pi^{\alpha}_{N,c}(f)}_{L^2(0,1)}\\&\leq&2\norm{f-g}_{L^2(0,1)}+\norm{g-\pi^{\alpha}_{N,c}(g)}_{L^2(0,1)}
	\end{eqnarray*}
	We recall that $g=\displaystyle\sum_{j=1}^{N}\scal{f,\phi_j}\phi_j$, then by using the Perseval's formula and \eqref{e2}, we get the following estimate
	\begin{eqnarray*}
	\norm{f-g}^2_{L^2(0,1)}&=&\sum_{j=N+1}^{\infty}\scal{f,\phi_j}^2\\&\leq& \frac{1}{\lambda^{2r}_{N+1}}\sum_{j=N+1}^{\infty}\lambda_j^{2r}\scal{f,\phi_j}^2\leq \frac{1}{\lambda^{2r}_{N+1}}\norm{f}^2_{\mathcal{H}^r_{\alpha}\left(0,1\right)}.
	\end{eqnarray*}
For the second part of our estimate, we use the Cauchy-Schwartz inequality with again \eqref{eq0.9}, we obtain
\begin{eqnarray*}
\norm{g-\pi^{\alpha}_{N,c}(g)}_{L^2(0,1)}&=&\norm{\displaystyle\sum_{j=1}^{N}\scal{f,\phi_j}\left(\phi_j-\pi^{\alpha}_{N,c}(\phi_j)\right)}_{L^2(0,1)}\\&\leq&\sum_{j=1}^{N}\left|\scal{f,\phi_j}\right|\norm{\phi_j-\pi^{\alpha}_{N,c}(\phi_j)}_{L^2(0,1)}\\&\leq&\sum_{j=1}^{N}\left|\scal{f,\phi_j}\right|\left[\sum_{n=N+1}^{\infty}\left|\scal{\vp,\phi_j}\right|^2\right]^{1/2}\\&\leq&\frac{C_{\alpha}}{N^{1/2}}e^{-aN}\norm{f}_{L^2(0,1)}.
\end{eqnarray*}	
\end{proof}
	Let $p\geq 1$ and $I=(0,1)$, we define the Hankel Sobolev space $\mathcal{H}^{p,r}_{\alpha}\left(I\right)$ the extension of $\mathcal{H}^r_{\alpha}\left(I\right)=\mathcal{H}^{2,r}_{\alpha}\left(I\right)$ over the $L^p$ space, that is
	\begin{equation}
	\mathcal{H}^{p,r}_{\alpha}\left(I\right)=\{u\in L^p(I);~~~ x^{(\alpha+k+1/2)}\left(\frac{1}{x}\frac{d}{dx}\right)^{(k)}\left(x^{-(\alpha+1/2)}u\right)\in L^p(I)~\,,\forall\, 0\leq k\leq r\}.
	\end{equation}
	It was proven in \cite{Boulsane-Jaming-Souabni} that $\pi^{\alpha}_{N,c}$ is uniformly bounded on $L^p(I)$ if and only if $4/3<p<4$. Moreover,
	we can easily guarantee that $L^2(I)\subset L^p(I)$ using the H\"older inequality for $1\leq p<2$, that is,
	for every $f\in L^2(I)$, we have \, $\norm{f}_{L^p(I)}\leq \left(\mu(I)\right)^{\frac{1}{p}-\frac{1}{2}}\norm{f}_{L^2(I)}.$
	From the previous two informations and by density with the previous theorem, we can extend the quality of convergence speed of \eqref{e15} over $\mathcal{H}^{p,r}_{\alpha}\left(I\right).$ Indeed, let $f\in\mathcal{H}^{p,r}_{\alpha}\left(I\right)$ and $(f_n)_n\subset \mathcal{C}_c^\infty(I)$ such that $f_n\to f$ when $n\to\infty$ in $\mathcal{H}^{p,r}_{\alpha}$ norm. Then for every $n\geq 0,$ we have
	\begin{eqnarray*}
	\norm{f-\pi^{\alpha}_{N,c}(f)}_{L^p(0,1)}&\leq& C\norm{f-f_n}_{L^p(0,1)}+\norm{f_n-\pi^{\alpha}_{N,c}(f_n)}_{L^p(0,1)}\\&\leq&C\norm{f-f_n}_{L^p(0,1)}+\norm{f_n-\pi^{\alpha}_{N,c}(f_n)}_{L^2(0,1)}\\&\leq&C\norm{f-f_n}_{L^p(0,1)}+\frac{C^1_{\alpha}}{\lambda^{r}_{N+1}}\norm{f_n}_{\mathcal{H}^r_{\alpha}\left(0,1\right)}+\frac{C^2_{\alpha}}{N^{1/2}}e^{-aN}\norm{f_n}_{L^2(0,1)}.
	\end{eqnarray*}
Let $g\in\mathcal{C}_c^\infty(I),$ then the series expansion of $g$ in CPSWFs basis is given by $g=\displaystyle\sum_{n=0}^{\infty}\scal{g,\vp}_{L^2(0,1)}\vp.$ From \cite{Boulsane-Jaming-Souabni}, we have for every $1\leq q\leq\infty,$ there exist $0\leq\gamma_q<1$ such that $\norm{\vp}_{L^q(0,1)}\lesssim n^{\gamma_q}$. Moreover, from the H\"older inequality and the fact that $\mathcal{L}_c^{\alpha}$ is self adjoint, we obtain
\begin{eqnarray*} \norm{g}^2_{L^2(0,1)}&=&\displaystyle\sum_{n=0}^{\infty}\scal{g,\vp}^2_{L^2(0,1)}\\&=&\displaystyle\sum_{n=0}^{\infty}\frac{1}{\left(\chi_n^{\alpha}(c)\right)^2}\scal{\mathcal{L}_c^{\alpha}g,\vp}^2_{L^2(0,1)}\\&\lesssim& \left[\displaystyle\sum_{n=0}^{\infty}n^{2(\gamma_q-2)}\right]\norm{\mathcal{L}_c^{\alpha}g}^2_{L^p(0,1)}\lesssim \norm{g}^2_{L^p(0,1)}.\end{eqnarray*}
It follows that
\begin{equation*}
\norm{f-\pi^{\alpha}_{N,c}(f)}_{L^p(0,1)}\leq C\norm{f-f_n}_{L^p(0,1)}+\frac{C^1_{\alpha}}{\lambda^{r}_{N+1}}\norm{f_n}_{\mathcal{H}^{p,r}_{\alpha}\left(0,1\right)}+\frac{C^2_{\alpha}}{N^{1/2}}e^{-aN}\norm{f_n}_{L^p(0,1)}.
\end{equation*}
By tending $n\to\infty$, we obtain
\begin{equation*}
\norm{f-\pi^{\alpha}_{N,c}(f)}_{L^p(0,1)}\leq\frac{C^1_{\alpha}}{\lambda^{r}_{N+1}}\norm{f}_{\mathcal{H}^{p,r}_{\alpha}\left(0,1\right)}+\frac{C^2_{\alpha}}{N^{1/2}}e^{-aN}\norm{f}_{L^p(0,1)}.
\end{equation*}
Finally, by duality we get our extension theorem.
\begin{corollary}
	Let $c>0$ and $\alpha>0$ two real numbers and $4/3<p<4,$ then for every $N>\frac{ec}{4}$ and $f\in\mathcal{H}^{p,r}_{\alpha}(0,1)$, we have
	\begin{equation}
	\norm{f-\pi^{\alpha}_{N,c}(f)}_{L^p(0,1)}\leq\frac{C^1_{\alpha}}{\lambda^{r}_{N+1}}\norm{f}_{\mathcal{H}^{p,r}_{\alpha}\left(0,1\right)}+\frac{C^2_{\alpha}}{N^{1/2}}e^{-aN}\norm{f}_{L^p(0,1)}.
	\end{equation}
\end{corollary}

\end{document}